\documentclass[12pt,twoside,reqno]{amsart}
\usepackage[colorlinks=false,citecolor=blue]{hyperref}
\usepackage{mathptmx, amsmath, amssymb, amsfonts, amsthm, mathptmx, enumerate, color,mathrsfs}
\usepackage{graphicx}
\newcommand{\vertiii}[1]{{\left\vert\kern-0.25ex\left\vert\kern-0.25ex\left\vert #1 
    \right\vert\kern-0.25ex\right\vert\kern-0.25ex\right\vert}}

\usepackage{epstopdf}
\newtheorem{theorem}{Theorem}[section]
\newtheorem{lemma}{Lemma}[section]

\newtheorem{corollary}{Corollary}[section]

\newtheorem{proposition}{Proposition}[section]
\numberwithin{equation}{section}

\begin{document}
\title{Integral Numerical Radius  and Operator Matrix Bounds}
\author{Shiva Sheybani, Hamid Reza Moradi, and Mohammad Sababheh}


\begin{abstract}
We establish new integral inequalities for the numerical radius and the operator norm of bounded linear operators on Hilbert spaces. Our results refine classical triangle-type and operator matrix inequalities by incorporating convex combinations and integral averaging techniques. Several consequences, including new identities, sharper bounds, and equality conditions, are obtained, revealing deeper structural connections between the numerical radius and operator norm.
 
\end{abstract}
\maketitle
%------------------------------------------------------------------------------------%
\pagestyle{myheadings}
\markboth{\centerline{}}
{\centerline{}}
\bigskip
\bigskip
%------------------------------------------------------------------------------------%
\section{Introduction}

The numerical radius is one of the most important quantities associated with bounded linear operators on Hilbert spaces. Let $\mathbb{B}(\mathbb{H})$ denote the $C^*$- algebra of all bounded linear operators on a complex Hilbert space $\mathbb{H}$. For $A\in\mathbb{B}(\mathbb{H})$, the \emph{numerical radius} of $A$ is defined by
\[
\omega(A)=\sup_{\|x\|=1}|\langle Ax,x\rangle|,
\]
while the usual \emph{operator norm} is given by
\[
\|A\|=\sup_{\|x\|=1}\|Ax\|.
\]
The numerical radius defines a norm on $\mathbb{B}(\mathbb{H})$ that is equivalent to the operator norm. The classical inequalities \cite[Theorem 1.3-1]{Gustafson_Book_1997}
\[
\frac{1}{2}\|A\|\le \omega(A)\le \|A\|
\]
hold for all $A\in\mathbb{B}(\mathbb{H})$, and they are sharp. In particular, if $A$ is normal, then $\omega(A)=\|A\|$. These relations make the numerical radius a fundamental tool in operator theory, matrix analysis, and applications where quadratic forms naturally arise.

Because of its geometric and spectral significance, the numerical radius has been studied extensively. Numerous refinements and improvements of inequalities involving $\omega(\cdot)$ and the operator norm have appeared in the literature. Researchers have investigated additive, multiplicative, and mixed inequalities, as well as bounds involving operator matrices, real and imaginary parts, and unitarily invariant norms. Such refinements often lead to sharper estimates, deeper structural insight, and improved tools for applications in perturbation theory, numerical linear algebra, and quantum information theory. The reader is referred to \cite {Abu-Omar_Rocky_2015,Abu-Omar_LAA_2015,Hirzallah_IEOT_2011,Hirzallah_MathScandin_2014,Hirzallah_NFAO_2011,Kittaneh_Studia_2003,Kittaneh_Studia_2005,Kittaneh_LAA_2015,Moradi_LAMA_2021,Moslehian_MathScand_2017,Sababheh_LAMA_2021,Sattari_LAA_2015} to mention a few.

In recent years, a growing body of research has focused on refining classical inequalities via integral representations and operator-matrix techniques \cite{Moradi_LAMA_2021, Sababheh_LAMA_2021}. These approaches allow one to bridge the gap between norm estimates and numerical radius bounds while capturing intermediate operator behavior. In particular, inequalities derived from convex combinations and integral averaging have proved effective in producing tighter bounds than classical triangle-type estimates.

Motivated by these developments, the present paper establishes new integral inequalities for the numerical radius and operator norm that refine known bounds and reveal new relationships between them. Our approach is inspired by the scalar inequality \cite[Theorem 2.1]{1}
\begin{equation}\label{1}
|a+b|\le 2\int_0^1 |(1-t)a+tb|\,dt \le |a|+|b|, a,b\in \mathbb{C},
\end{equation}
and extends it to operators through numerical radius techniques.

We first prove an integral refinement of the numerical radius triangle inequality
\[
\omega(A+B)\le 2\int_0^1 \omega((1-t)A+tB)\,dt \le \omega(A)+\omega(B),
\]
which provides a continuous interpolation between $\omega(A+B)$ and the classical bound $\omega(A)+\omega(B)$. As a consequence, we derive bounds involving the real and imaginary parts of operators and obtain estimates expressed in terms of operator matrices.

Using block operator techniques, we establish inequalities of the form
\[
\|A+B\|\le 2\int_0^1
\omega\!\left(
\begin{bmatrix}
0 & (1-t)A+tB\\
tA^*+(1-t)B^* & 0
\end{bmatrix}
\right)dt,
\]
which connect the norm of sums to numerical radii of associated operator matrices and yield refined bounds that recover classical estimates as special cases.

Our main result provides a refined inequality combining numerical radius and operator norm estimates:
\[
\begin{aligned}
  & \left\| A+B \right\|
  +\left| \int_{0}^{1}\left(
  \omega \left(
  \begin{matrix}
   O & (1-t)A+tB  \\
   tA^{*}+(1-t)B^{*} & O
  \end{matrix}
  \right)
  -\left\| (1-t)A+tB \right\|
  \right)dt \right| \\
 & \le \int_{0}^{1}\left(
 \omega \left(
 \begin{matrix}
   O & (1-t)A+tB  \\
   tA^{*}+(1-t)B^{*} & O
 \end{matrix}
 \right)
 +\left\| (1-t)A+tB \right\|
 \right)dt \\
 & \le \left\| A \right\|+\left\| B \right\|,
\end{aligned}
\]
which yields sharper control than classical norm inequalities. In this context, if $S\in\mathbb{B}(\mathbb{H}), S^*$ is the adjoint of $S$. As an application, we obtain the refined bound
{\small
\[\frac{1}{2}\left\| A \right\|+\frac{1}{2}\left| \int\limits_{0}^{1}{\left( \omega \left( \begin{matrix}
   O & \left( 1-t \right)\Re A+it\Im A  \\
   t\Re A+i\left( 1-t \right){\Im A} & O  \\
\end{matrix} \right)-\left\| \left( 1-t \right)\Re A+it\Im A \right\| \right)dt} \right|\le \omega \left( A \right),\]
}
providing a new lower estimate for the numerical radius in terms of integral averages of operator matrices. Here $\Re A$ and $\Im A$ refer to the real and imaginary parts of $A$, respectively. The obtained bounds sharpen classical estimates, unify several known inequalities, and provide new tools for further developments in operator theory and matrix analysis.

\section{Main results}

The numerical radius satisfies the classical triangle inequality
\[
\omega(A+B)\le \omega(A)+\omega(B),
\]
which provides a basic estimate for sums of operators. However, this bound does not capture how the operators interact along intermediate convex combinations. Inspired by integral refinements of the scalar triangle inequality, it is natural to investigate whether averaging the numerical radius over convex combinations of $A$ and $B$ can yield a sharper and more informative bound.

The next result provides an integral refinement of the numerical radius triangle inequality. It interpolates between $\omega(A+B)$ and $\omega(A)+\omega(B)$, providing a continuous measure of the interaction between $A$ and $B$ and yielding tighter estimates that will be useful in subsequent developments.

\begin{theorem}\label{6}
Let $A,B\in \mathbb B\left( \mathbb H \right)$. Then
\[\omega \left( A+B \right)\le 2\int\limits_{0}^{1}{\omega \left( \left( 1-t \right)A+tB \right)dt}\le \omega \left( A \right)+\omega \left( B \right).\]
\end{theorem}
\begin{proof}
Let $x\in\mathbb{H}$, and put $a=\left\langle Ax,x \right\rangle $ and $b=\left\langle Bx,x \right\rangle $ in \eqref{1} to get
	\[\begin{aligned}
   \left| \left\langle \left( A+B \right)x,x \right\rangle  \right|&=\left| \left\langle Ax,x \right\rangle +\left\langle Bx,x \right\rangle  \right| \\ 
 & \le 2\int\limits_{0}^{1}{\left| \left( 1-t \right)\left\langle Ax,x \right\rangle +t\left\langle Bx,x \right\rangle  \right|dt} \\ 
 & =2\int\limits_{0}^{1}{\left| \left\langle \left( \left( 1-t \right)A+tB \right)x,x \right\rangle  \right|dt} \\ 
 & \le \left| \left\langle Ax,x \right\rangle  \right|+\left| \left\langle Bx,x \right\rangle  \right|. 
\end{aligned}\]
That is, if $x\in\mathbb{H}$, then
\begin{equation}\label{2}
\left| \left\langle \left( A+B \right)x,x \right\rangle  \right|\le 2\int\limits_{0}^{1}{\left| \left\langle \left( \left( 1-t \right)A+tB \right)x,x \right\rangle  \right|dt}\le \left| \left\langle Ax,x \right\rangle  \right|+\left| \left\langle Bx,x \right\rangle  \right|.
\end{equation}
After taking the supremum over all unit vectors $x\in \mathbb H$, in \eqref{2}, we obtain
	\[\omega \left( A+B \right)\le 2\int\limits_{0}^{1}{\omega \left( \left( 1-t \right)A+tB \right)dt}.\]
On the other hand, since 
	\[\omega \left( \left( 1-t \right)A+tB \right)\le \left( 1-t \right)\omega \left( A \right)+t\omega \left( B \right),\]
we get after taking the integral over $0\le t\le 1$, 
	\[\int\limits_{0}^{1}{\omega \left( \left( 1-t \right)A+tB \right)dt}\le \frac{\omega \left( A \right)+\omega \left( B \right)}{2}.\]
Indeed, we have shown that
\[\omega \left( A+B \right)\le 2\int\limits_{0}^{1}{\omega \left( \left( 1-t \right)A+tB \right)dt}\le \omega \left( A \right)+\omega \left( B \right),\]
as required.
\end{proof}
%%%%%%%%%%%%%%%%%%%%%%%%%%%%%%%%%%%%%%%%%%%%%%%%%%%%%%%%%%%%%%%%%%%%%%%%%%%%%%%%%%%%%%%%%%%%%%%%%%%%%%%%%%%%%%%%%%%%%%%%%%%%%%%%%%%%%%%%%%%%%%%%%%%%%%%%%%%%%%%%%%%%%%%%%%%%%%%%%%%%%%%%%%%%%%%%%%%%%%%%%%%%%%%%%%%%%%%%%%%%%%%%%%%%%%%%%%%%%%%%%%%%%%%%%%%%%%%%%%%%%%%%%%%%%%%%%%%%%%%%%%%%%%%%%%%%%%%%%%%%%%%%%%%%%%%%%%%%%%%%%%%%%%%%%%%%%%%%%%%%%%%%%%%%%%%%%%%%%%%%%%%%%%%%%%%%%%%%%%%%%%%%%%%%%%%%%%%%%%%%%%%%%%%%%%%%%%%%%%%%%%%%%%%%%%%%%%%%%%%%%%%%%%%%%%%%%%%%%%%%%%%%%%%%%%%%%%%%%%%%%%%%%%%%%%%%%%%%%%%%%%%%%%%%%%%%%%%%%%%%%%%%%%%%%%%%%%%%%%%%%%%%%%%%%%%%%%%%%%%%%%%%%%%%%%%%%%%%%%%%%%%%%%%%%%%%%%%%%%%%%%%%%%%%%%%%%%%%%%%%%%%%%%%%

It is well known that the numerical radius controls the real part of an operator through the inequality
\[
\|\Re A\|\le \omega(A),
\]
which follows from the definition of the numerical radius. Although this estimate is fundamental, it does not reveal how $A$ and $A^*$ interact through intermediate combinations. By applying the integral refinement of Theorem \ref{6} to the pair $(A,A^*)$, we obtain the following result, which refines the classical bound by inserting an averaged numerical radius term between $\|\Re A\|$ and $\omega(A)$.

\begin{corollary}\label{3}
Let $A\in \mathbb B\left( \mathbb H \right)$. Then
\[\left\| \Re A \right\|\le \int\limits_{0}^{1}{\omega \left( \left( 1-t \right)A+t{{A}^{*}} \right)dt}\le \omega \left( A \right).\]
\end{corollary}
\begin{proof}
Applying Theorem \ref{6}, we have
\[\begin{aligned}
   2\left\| \Re A \right\|&=\omega \left( A+{{A}^{*}} \right) \\ 
 & \le 2\int\limits_{0}^{1}{\omega \left( \left( 1-t \right)A+t{{A}^{*}} \right)dt} \\ 
 & \le \omega \left( A \right)+\omega \left( {{A}^{*}} \right) \\ 
 & =2\omega \left( A \right),  
\end{aligned}\]
as required.
\end{proof}
%%%%%%%%%%%%%%%%%%%%%%%%%%%%%%%%%%%%%%%%%%%%%%%%%%%%%%%%%%%%%%%%%%%%%%%%%%%%%%%%%%%%%%%%%%%%%%%%%%%%%%%%%%%%%%%%%%%%%%%%%%%%%%%%%%%%%%%%%%%%%%%%%%%%%%%%%%%%%%%%%%%%%%%%%%%%%%%%%%%%%%%%%%%%%%%%%%%%%%%%%%%%%%%%%%%%%%%%%%%%%%%%%%%%%%%%%%%%%%%%%%%%%%%%%%%%%%%%%%%%%%%%%%%%%%%%%%%%%%%%%%%%%%%%%%%%%%%%%%%%%%%%%%%%%%%%%%%%%%%%%%%%%%%%%%%%%%%%%%%%%%%%%%%%%%%%%%%%%%%%%%%%%%%%%%%%%%%%%%%%%%%%%%%%%%%%%%%%%%%%%%%%%%%%%%%%%%%%%%%%%%%%%%%%%%%%%%%%%%%%%%%%%%%%%%%%%%%%%%%%%%%%%%%%%%%%%%%%%%%%%%%%%%%%%%%%%%%%%%%%%%%%%%%%%%%%%%%%%%%%%%%%%%%%%%%%%%%%%%%%%%%%%%%%%%%%%%%%%%%%%%%%%%%%%%%%%%%%%%%%%%%%%%%%%%%%%%%%%%%%%%%%%%%%%%%%%%%%%%%%%%%%%%%%

It is known that the norm of a sum of operators can be estimated via operator matrices by \cite{Hirzallah_IEOT_2011}
\[
\|S+T\|\le 2\,\omega\!\left(
\begin{bmatrix}
O & S\\
T^{*} & O
\end{bmatrix}
\right),
\]
which connects the operator norm to the numerical radius of an associated block operator matrix. Although this bound is effective, it does not capture intermediate interactions between $S$ and $T$. By incorporating an integral average over convex combinations of $S$ and $T$, the following theorem refines this estimate by inserting a tighter intermediate bound between $\|S+T\|$ and the classical operator matrix bound.

\begin{theorem}\label{Thm_S_T_Main}
Let $S,T\in \mathbb B\left( \mathbb H \right)$. Then \[\left\| S+T \right\|\le 2\int\limits_{0}^{1}{\omega \left( \left[ \begin{matrix} O & \left( 1-t \right)S+tT \\ t{{S}^{*}}+\left( 1-t \right){{T}^{*}} & O \\ \end{matrix} \right] \right)dt}\le 2\omega \left( \left[ \begin{matrix} O & S \\ {{T}^{*}} & O \\ \end{matrix} \right] \right).\]
\end{theorem}
\begin{proof}
Let  $A=\left[ \begin{matrix}
   O & S  \\
   {{T}^{*}} & O  \\
\end{matrix} \right]\in\mathbb{B}(\mathbb H\oplus \mathbb H)$. By Corollary \ref{3}, we have
\[\begin{aligned}
   \frac{1}{2}\left\| S+T \right\|&=\frac{1}{2}\left\| \left[ \begin{matrix}
   O & S+T  \\
   {{S}^{*}}+{{T}^{*}} & O  \\
\end{matrix} \right] \right\| \\ 
 & =\frac{1}{2}\left\| \left[ \begin{matrix}
   O & S  \\
   {{T}^{*}} & O  \\
\end{matrix} \right]+\left[ \begin{matrix}
   O & T  \\
   {{S}^{*}} & O  \\
\end{matrix} \right] \right\| \\ 
 & =\left\| \Re A \right\| \\ 
 & \le \int\limits_{0}^{1}{\omega \left( \left( 1-t \right)A+t{{A}^{*}} \right)dt} \\ 
 & =\int\limits_{0}^{1}{\omega \left( \left( 1-t \right)\left[ \begin{matrix}
   O & S  \\
   {{T}^{*}} & O  \\
\end{matrix} \right]+t\left[ \begin{matrix}
   O & T  \\
   {{S}^{*}} & O  \\
\end{matrix} \right] \right)dt} \\ 
 & =\int\limits_{0}^{1}{\omega \left( \left[ \begin{matrix}
   O & \left( 1-t \right)S+tT  \\
   t{{S}^{*}}+\left( 1-t \right){{T}^{*}} & O  \\
\end{matrix} \right] \right)dt} \\ 
 & \le \omega \left( A \right) \\ 
 & =\omega \left( \left[ \begin{matrix}
   O & S  \\
   {{T}^{*}} & O  \\
\end{matrix} \right] \right),  
\end{aligned}\]
as required.
\end{proof}

%%%%%%%%%%%%%%%%%%%%%%%%%%%%%%%%%%%%%%%%%%%%%%%%%%%%%%%%%%%%%%%%%%%%%%%%%%%%%%%%%%%%%%%%%%%%%%%%%%%%%%%%%%%%%%%%%%%%%%%%%%%%%%%%%%%%%%%%%%%%%%%%%%%%%%%%%%%%%%%%%%%%%%%%%%%%%%%%%%%%%%%%%%%%%%%%%%%%%%%%%%%%%%%%%%%%%%%%%%%%%%%%%%%%%%%%%%%%%%%%%%%%%%%%%%%%%%%%%%%%%%%%%%%%%%%%%%%%%%%%%%%%%%%%%%%%%%%%%%%%%%%%%%%%%%%%%%%%%%%%%%%%%%%%%%%%%%%%%%%%%%%%%%%%%%%%%%%%%%%%%%%%%%%%%%%%%%%%%%%%%%%%%%%%%%%%%%%%%%%%%%%%%%%%%%%%%%%%%%%%%%%%%%%%%%%%%%%%%%%%%%%%%%%%%%%%%%%%%%%%%%%%%%%%%%%%%%%%%%%%%%%%%%%%%%%%%%%%%%%%%%%%%%%%%%%%%%%%%%%%%%%%%%%%%%%%%%%%%%%%%%%%%%%%%%%%%%%%%%%%%%%%%%%%%%%%%%%%%%%%%%%%%%%%%%%%%%%%%%%%%%%%%%%%%%%%%%%%%%%%%%%%%%%%
The numerical radius of an operator can be expressed in terms of block operator matrices, which often provide sharper insight than direct estimates. In particular, embedding $A$ into a $2\times2$ operator matrix allows one to relate $\omega(A)$ to numerical radii of off–diagonal operator structures. The following corollary shows that $\omega(A)$ admits an integral upper bound obtained from such block matrices, yielding a refined representation. Moreover, when $A$ is normal, the integral expression collapses to $\|A\|$, revealing an exact norm identity and highlighting the sharpness of the estimate.

\begin{corollary}\label{4}
Let $A\in \mathbb B\left( \mathbb H \right)$. Then
\[\omega \left( A \right)\le 2\int\limits_{0}^{1}{\omega \left( \left[ \begin{matrix}
   O & \left( 1-t \right)A  \\
   tA & O  \\
\end{matrix} \right] \right)dt}.\]
In particular, if $A$ is a normal operator, then
\[2\int\limits_{0}^{1}{\omega \left( \left[ \begin{matrix}
   O & \left( 1-t \right)A  \\
   tA & O  \\
\end{matrix} \right] \right)dt}=\left\| A \right\|.\]
\end{corollary}
\begin{proof}
Replace $A$ by ${{e}^{i\theta }}A$, in Corollary \ref{3}, we get
\[\begin{aligned}
   \left\| \Re\left( {{e}^{i\theta }}A \right) \right\|&\le \int\limits_{0}^{1}{\omega \left( \left( 1-t \right){{e}^{i\theta }}A+t{{e}^{-i\theta }}{{A}^{*}} \right)dt} \\ 
 & \le \int\limits_{0}^{1}{\left\| \left( 1-t \right){{e}^{i\theta }}A+t{{e}^{-i\theta }}{{A}^{*}} \right\|dt} \\ 
 & \le \int\limits_{0}^{1}{\left\| \left( 1-t \right){{e}^{i\theta }}A+t{{e}^{-i\theta }}{{A}^{*}} \right\|dt} \\ 
 & \le 2\int\limits_{0}^{1}{\omega \left( \left[ \begin{matrix}
   O & \left( 1-t \right)A  \\
   tA & O  \\
\end{matrix} \right] \right)dt}.  
\end{aligned}\]
That is,
\[\left\| \Re\left( {{e}^{i\theta }}A \right) \right\|\le 2\int\limits_{0}^{1}{\omega \left( \left[ \begin{matrix}
   O & \left( 1-t \right)A  \\
   tA & O  \\
\end{matrix} \right] \right)dt}.\]
By taking the supremum over $\theta \in \mathbb{R}$, we deduce
\[\omega \left( A \right)\le 2\int\limits_{0}^{1}{\omega \left( \left[ \begin{matrix}
   O & \left( 1-t \right)A  \\
   tA & O  \\
\end{matrix} \right] \right)dt},\]
where we have used the well-known identity $\omega(A)=\sup_{\theta\in\mathbb{R}}\left\|\Re\left(e^{i\theta}A\right)\right\|.$ This completes the proof.
\end{proof}
%%%%%%%%%%%%%%%%%%%%%%%%%%%%%%%%%%%%%%%%%%%%%%%%%%%%%%%%%%%%%%%%%%%%%%%%%%%%%%%%%%%%%%%%%%%%%%%%%%%%%%%%%%%%%%%%%%%%%%%%%%%%%%%%%%%%%%%%%%%%%%%%%%%%%%%%%%%%%%%%%%%%%%%%%%%%%%%%%%%%%%%%%%%%%%%%%%%%%%%%%%%%%%%%%%%%%%%%%%%%%%%%%%%%%%%%%%%%%%%%%%%%%%%%%%%%%%%%%%%%%%%%%%%%%%%%%%%%%%%%%%%%%%%%%%%%%%%%%%%%%%%%%%%%%%%%%%%%%%%%%%%%%%%%%%%%%%%%%%%%%%%%%%%%%%%%%%%%%%%%%%%%%%%%%%%%%%%%%%%%%%%%%%%%%%%%%%%%%%%%%%%%%%%%%%%%%%%%%%%%%%%%%%%%%%%%%%%%%%%%%%%%%%%%%%%%%%%%%%%%%%%%%%%%%%%%%%%%%%%%%%%%%%%%%%%%%%%%%%%%%%%%%%%%%%%%%%%%%%%%%%%%%%%%%%%%%%%%%%%%%%%%%%%%%%%%%%%%%%%%%%%%%%%%%%%%%%%%%%%%%%%%%%%%%%%%%%%%%%%%%%%%%%%%%%%%%%%%%%%%%%%%%%%%
The numerical radius admits several characterizations in terms of the real parts of rotated operators, most notably
\(\omega(A)=\sup_{\theta\in\mathbb{R}}\|\Re(e^{i\theta}A)\|\).
The following result provides a new identity for the numerical radius, obtained via an integral averaging process involving $A$ and its adjoint. This formulation offers a fresh perspective by combining rotational invariance with convex interpolation between $A$ and $A^*$, thereby revealing additional structure behind the numerical radius. In the special case where $A$ is normal, the identity reduces to an exact representation of the operator norm, highlighting both the sharpness and the structural significance of the formula.
\begin{corollary}
Let $A\in \mathbb B\left( \mathbb H \right)$. Then
\[\omega \left( A \right)=\underset{\theta \in \mathbb{R}}{\mathop{\sup }}\,\int\limits_{0}^{1}{\omega \left( \left( 1-t \right){{e}^{i\theta }}A+t{{e}^{-i\theta }}{{A}^{*}} \right)dt}.\]
In particular, if $A$ is a normal operator, then
\[\left\| A \right\|=\underset{\theta \in \mathbb{R}}{\mathop{\sup }}\,\int\limits_{0}^{1}{\omega \left( \left( 1-t \right){{e}^{i\theta }}A+t{{e}^{-i\theta }}{{A}^{*}} \right)dt}.\]
\end{corollary}
\begin{proof}
It follows from Corollary \ref{4} that
	\[\left\| \Re \left( {{e}^{i\theta }}A \right) \right\|\le \int\limits_{0}^{1}{\omega \left( \left( 1-t \right){{e}^{i\theta }}A+t{{e}^{-i\theta }}{{A}^{*}} \right)dt}\le \omega \left( A \right).\]
Since ${{\sup }_{\theta \in \mathbb{R}}}\left\| \Re\left( {{e}^{i\theta }}T \right) \right\|=\omega \left( T \right)$, we obtain the desired result.
\end{proof}
%%%%%%%%%%%%%%%%%%%%%%%%%%%%%%%%%%%%%%%%%%%%%%%%%%%%%%%%%%%%%%%%%%%%%%%%%%%%%%%%%%%%%%%%%%%%%%%%%%%%%%%%%%%%%%%%%%%%%%%%%%%%%%%%%%%%%%%%%%%%%%%%%%%%%%%%%%%%%%%%%%%%%%%%%%%%%%%%%%%%%%%%%%%%%%%%%%%%%%%%%%%%%%%%%%%%%%%%%%%%%%%%%%%%%%%%%%%%%%%%%%%%%%%%%%%%%%%%%%%%%%%%%%%%%%%%%%%%%%%%%%%%%%%%%%%%%%%%%%%%%%%%%%%%%%%%%%%%%%%%%%%%%%%%%%%%%%%%%%%%%%%%%%%%%%%%%%%%%%%%%%%%%%%%%%%%%%%%%%%%%%%%%%%%%%%%%%%%%%%%%%%%%%%%%%%%%%%%%%%%%%%%%%%%%%%%%%%%%%%%%%%%%%%%%%%%%%%%%%%%%%%%%%%%%%%%%%%%%%%%%%%%%%%%%%%%%%%%%%%%%%%%%%%%%%%%%%%%%%%%%%%%%%%%%%%%%%%%%%%%%%%%%%%%%%%%%%%%%%%%%%%%%%%%%%%%%%%%%%%%%%%%%%%%%%%%%%%%%%%%%%%%%%%%%%%%%%%%%%%%%%%%%%%%

The numerical radius is closely related to the real and imaginary parts of an operator, and many of its estimates arise from expressions involving $A$ and $A^*$. While classical bounds control these components separately, they do not reflect the contributions of symmetric and skew-symmetric combinations to the numerical radius. The following corollary provides a new integral estimate that bounds both the symmetric and skew–symmetric interpolations between $A$ and $A^*$ by $\omega(A)$. This unified bound offers additional insight into the structure of the numerical radius and complements its known representations in terms of real and imaginary parts.

\begin{corollary}
Let $A\in \mathbb B\left( \mathbb H \right)$. Then
\[\max \left\{ \int\limits_{0}^{1}{\omega \left( \left( 1-t \right)A-t{{A}^{*}} \right)dt},\int\limits_{0}^{1}{\omega \left( \left( 1-t \right)A+t{{A}^{*}} \right)dt} \right\}\le \omega \left( A \right).\]
\end{corollary}
\begin{proof}
In Theorem \ref{Thm_S_T_Main}, if we replace $T$ by $-{{S}^{*}}$, we obtain
	\[\begin{aligned}
   2\left\| \Im S \right\|&=\left\| S-{{S}^{*}} \right\| \\ 
 & \le 2\int\limits_{0}^{1}{\omega \left( \left[ \begin{matrix}
   O & \left( 1-t \right)S-t{{S}^{*}}  \\
   t{{S}^{*}}-\left( 1-t \right)S & O  \\
\end{matrix} \right] \right)dt} \\ 
 & =2\int\limits_{0}^{1}{\omega \left( \left[ \begin{matrix}
   O & \left( 1-t \right)S-t{{S}^{*}}  \\
   -\left( \left( 1-t \right)S-t{{S}^{*}} \right) & O  \\
\end{matrix} \right] \right)dt} \\ 
 & =2\int\limits_{0}^{1}{\omega \left( \left[ \begin{matrix}
   O & \left( 1-t \right)S-t{{S}^{*}}  \\
   \left( 1-t \right)S-t{{S}^{*}} & O  \\
\end{matrix} \right] \right)dt} \\ 
 & =2\int\limits_{0}^{1}{\omega \left( \left( 1-t \right)S-t{{S}^{*}} \right)dt} \\ 
 & \le 2\omega \left( \left[ \begin{matrix}
   O & S  \\
   -S & O  \\
\end{matrix} \right] \right) \\ 
 & =2\omega \left( \left[ \begin{matrix}
   O & S  \\
   S & O  \\
\end{matrix} \right] \right) \\ 
 & =2\omega \left( S \right).  
\end{aligned}\]
Thus, we have shown that
\[\left\| \Im S \right\|\le \int\limits_{0}^{1}{\omega \left( \left( 1-t \right)S-t{{S}^{*}} \right)dt}\le \omega \left( S \right).\]
If we replace $S$ by $A$ in the last inequality, we get 
\[\left\| \Im A \right\|\le \int\limits_{0}^{1}{\omega \left( \left( 1-t \right)A-t{{A}^{*}} \right)dt}\le \omega \left( A \right).\]
Combining this with Corollary \ref{3} implies the desired result.
\end{proof}
%%%%%%%%%%%%%%%%%%%%%%%%%%%%%%%%%%%%%%%%%%%%%%%%%%%%%%%%%%%%%%%%%%%%%%%%%%%%%%%%%%%%%%%%%%%%%%%%%%%%%%%%%%%%%%%%%%%%%%%%%%%%%%%%%%%%%%%%%%%%%%%%%%%%%%%%%%%%%%%%%%%%%%%%%%%%%%%%%%%%%%%%%%%%%%%%%%%%%%%%%%%%%%%%%%%%%%%%%%%%%%%%%%%%%%%%%%%%%%%%%%%%%%%%%%%%%%%%%%%%%%%%%%%%%%%%%%%%%%%%%%%%%%%%%%%%%%%%%%%%%%%%%%%%%%%%%%%%%%%%%%%%%%%%%%%%%%%%%%%%%%%%%%%%%%%%%%%%%%%%%%%%%%%%%%%%%%%%%%%%%%%%%%%%%%%%%%%%%%%%%%%%%%%%%%%%%%%%%%%%%%%%%%%%%%%%%%%%%%%%%%%%%%%%%%%%%%%%%%%%%%%%%%%%%%%%%%%%%%%%%%%%%%%%%%%%%%%%%%%%%%%%%%%%%%%%%%%%%%%%%%%%%%%%%%%%%%%%%%%%%%%%%%%%%%%%%%%%%%%%%%%%%%%%%%%%%%%%%%%%%%%%%%%%%%%%%%%%%%%%%%%%%%%%%%%%%%%%%%%%%%%%%%%%
The following proposition follows immediately from the Hermite-Hadamard inequality applied to the convex function $f(t)=\left\| \left( 1-t \right)A+tB \right\|.$
\begin{proposition}
Let $A,B\in \mathbb B\left( \mathbb H \right)$. Then
\[\left\| A+B \right\|\le 2\int\limits_{0}^{1}{\left\| \left( 1-t \right)A+tB \right\|dt}.\]
\end{proposition}

\begin{lemma}\label{5}
Let $S,T\in \mathbb B\left( \mathbb H \right)$. Then
\[\left\| S+T \right\|\le 2\min \left\{ \int\limits_{0}^{1}{\omega \left( \left[ \begin{matrix}
   O & \left( 1-t \right)S+tT  \\
   t{{S}^{*}}+\left( 1-t \right){{T}^{*}} & O  \\
\end{matrix} \right] \right)dt},\int\limits_{0}^{1}{\left\| \left( 1-t \right)S+tT \right\|dt} \right\}.\]
\end{lemma}
\begin{proof}
Let $a=\left\langle Ax,x \right\rangle $ and $b=\left\langle Bx,x \right\rangle $, in \eqref{1}. By applying the same method as in the proof of Theorem \ref{6}, we get the desired result.
\end{proof}
Classical estimates for sums of operators typically treat the operator norm and the numerical radius separately. For instance, the triangle inequality gives $\|S+T\|\le \|S\|+\|T\|$, while numerical radius bounds are often derived through associated operator matrices. However, these estimates do not quantify the gap between norm-based and numerical-radius-based bounds. The following theorem provides a refined inequality that simultaneously incorporates both quantities through an integral averaging process. By comparing the numerical radius of block operator matrices with the norms of convex combinations of $S$ and $T$, the result yields a sharper estimate for $\|S+T\|$. It offers a unified framework that bridges norm inequalities and numerical radius bounds.

\begin{theorem}\label{10}
Let $S,T\in \mathbb B\left( \mathbb H \right)$. Then
\[\begin{aligned}
  & \left\| S+T \right\|+\left| \int\limits_{0}^{1}{\left( \omega \left( \begin{matrix}
   O & \left( 1-t \right)S+tT  \\
   t{{S}^{*}}+\left( 1-t \right){{T}^{*}} & O  \\
\end{matrix} \right)-\left\| \left( 1-t \right)S+tT \right\| \right)dt} \right| \\ 
 & \le \int\limits_{0}^{1}{\left( \omega \left( \begin{matrix}
   O & \left( 1-t \right)S+tT  \\
   t{{S}^{*}}+\left( 1-t \right){{T}^{*}} & O  \\
\end{matrix} \right)+\left\| \left( 1-t \right)S+tT \right\| \right)dt} \\ 
 & \le \left\| S \right\|+\left\| T \right\|.
\end{aligned}\]
\end{theorem}
\begin{proof}
By Lemma \ref{5}, we have
\[\begin{aligned}
  & \left\| S+T \right\| \\ 
 & \le \int\limits_{0}^{1}{\omega \left( \left[ \begin{matrix}
   O & \left( 1-t \right)S+tT  \\
   t{{S}^{*}}+\left( 1-t \right){{T}^{*}} & O  \\
\end{matrix} \right] \right)dt}+\int\limits_{0}^{1}{\left\| \left( 1-t \right)S+tT \right\|dt} \\ 
 &\quad -\left| \int\limits_{0}^{1}{\omega \left( \left[ \begin{matrix}
   O & \left( 1-t \right)S+tT  \\
   t{{S}^{*}}+\left( 1-t \right){{T}^{*}} & O  \\
\end{matrix} \right] \right)dt}-\int\limits_{0}^{1}{\left\| \left( 1-t \right)S+tT \right\|dt} \right|.
\end{aligned}\]
So,
\begin{equation}\label{7}
\begin{aligned}
  & \left\| S+T \right\|+\left| \int\limits_{0}^{1}{\omega \left( \left[ \begin{matrix}
   O & \left( 1-t \right)S+tT  \\
   t{{S}^{*}}+\left( 1-t \right){{T}^{*}} & O  \\
\end{matrix} \right] \right)dt}-\int\limits_{0}^{1}{\left\| \left( 1-t \right)S+tT \right\|dt} \right| \\ 
 & \le \int\limits_{0}^{1}{\omega \left( \left[ \begin{matrix}
   O & \left( 1-t \right)S+tT  \\
   t{{S}^{*}}+\left( 1-t \right){{T}^{*}} & O  \\
\end{matrix} \right] \right)dt}+\int\limits_{0}^{1}{\left\| \left( 1-t \right)S+tT \right\|dt}. 
\end{aligned}
\end{equation}
We know that
\[\begin{aligned}
  & \omega \left( \left[ \begin{matrix}
   O & \left( 1-t \right)S+tT  \\
   t{{S}^{*}}+\left( 1-t \right){{T}^{*}} & O  \\
\end{matrix} \right] \right) \\ 
 & \le \frac{1}{2}\left( \left\| \left( 1-t \right)S+tT \right\|+\left\| t{{S}^{*}}+\left( 1-t \right){{T}^{*}} \right\| \right) \\ 
 & \le \frac{1}{2}\left( \left( 1-t \right)\left\| S \right\|+t\left\| T \right\|+t\left\| {{S}^{*}} \right\|+\left( 1-t \right)\left\| {{T}^{*}} \right\| \right) \\ 
 & =\frac{1}{2}\left( \left\| S \right\|+\left\| T \right\| \right).  
\end{aligned}\]
Thus,
\begin{equation}\label{8}
\int\limits_{0}^{1}{\omega \left( \left[ \begin{matrix}
   O & \left( 1-t \right)S+tT  \\
   t{{S}^{*}}+\left( 1-t \right){{T}^{*}} & O  \\
\end{matrix} \right] \right)dt}\le \frac{1}{2}\left( \left\| S \right\|+\left\| T \right\| \right).
\end{equation}
We also have
\[\left\| \left( 1-t \right)S+tT \right\|\le \left( 1-t \right)\left\| S \right\|+t\left\| T \right\|.\]
Hence
\begin{equation}\label{9}
\int\limits_{0}^{1}{\left\| \left( 1-t \right)S+tT \right\|dt}\le \int\limits_{0}^{1}{\left( \left( 1-t \right)\left\| S \right\|+t\left\| T \right\| \right)dt}=\frac{1}{2}\left( \left\| S \right\|+\left\| T \right\| \right).
\end{equation}
Now, combining two inequalities \eqref{8} and \eqref{9} together with \eqref{7} implies the desired result.
\end{proof}
%%%%%%%%%%%%%%%%%%%%%%%%%%%%%%%%%%%%%%%%%%%%%%%%%%%%%%%%%%%%%%%%%%%%%%%%%%%%%%%%%%%%%%%%%%%%%%%%%%%%%%%%%%%%%%%%%%%%%%%%%%%%%%%%%%%%%%%%%%%%%%%%%%%%%%%%%%%%%%%%%%%%%%%%%%%%%%%%%%%%%%%%%%%%%%%%%%%%%%%%%%%%%%%%%%%%%%%%%%%%%%%%%%%%%%%%%%%%%%%%%%%%%%%%%%%%%%%%%%%%%%%%%%%%%%%%%%%%%%%%%%%%%%%%%%%%%%%%%%%%%%%%%%%%%%%%%%%%%%%%%%%%%%%%%%%%%%%%%%%%%%%%%%%%%%%%%%%%%%%%%%%%%%%%%%%%%%%%%%%%%%%%%%%%%%%%%%%%%%%%%%%%%%%%%%%%%%%%%%%%%%%%%%%%%%%%%%%%%%%%%%%%%%%%%%%%%%%%%%%%%%%%%%%%%%%%%%%%%%%%%%%%%%%%%%%%%%%%%%%%%%%%%%%%%%%%%%%%%%%%%%%%%%%%%%%%%%%%%%%%%%%%%%%%%%%%%%%%%%%%%%%%%%%%%%%%%%%%%%%%%%%%%%%%%%%%%%%%%%%%%%%%%%%%%%%%%%%%%%%%%%%%%%%%

The following corollary provides a refined lower bound for the numerical radius by incorporating an integral comparison between the numerical radius of an associated operator matrix and the norm of convex combinations of the real and imaginary parts of $A$. This provides an interesting refinement of the celebrated bound $\frac{1}{2}\|A\|\leq \omega(A).$

\begin{corollary}\label{11}
Let $A\in \mathbb B\left( \mathbb H \right)$. Then
{\small
\[\frac{1}{2}\left\| A \right\|+\frac{1}{2}\left| \int\limits_{0}^{1}{\left( \omega \left( \begin{matrix}
   O & \left( 1-t \right)\Re A+it\Im A  \\
   t\Re A+i\left( 1-t \right){\Im A} & O  \\
\end{matrix} \right)-\left\| \left( 1-t \right)\Re A+it\Im A \right\| \right)dt} \right|\le \omega \left( A \right).\]
}
\end{corollary}
\begin{proof}
It follows from Theorem \ref{10} that
\[\begin{aligned}
  & \left\| S+iT \right\|+\left| \int\limits_{0}^{1}{\left( \omega \left( \begin{matrix}
   O & \left( 1-t \right)S+itT  \\
   t{{S}^{*}}+i\left( 1-t \right){{T}^{*}} & O  \\
\end{matrix} \right)-\left\| \left( 1-t \right)S+itT \right\| \right)dt} \right| \\ 
 & \le \int\limits_{0}^{1}{\left( \omega \left( \begin{matrix}
   O & \left( 1-t \right)S+itT  \\
   t{{S}^{*}}+i\left( 1-t \right){{T}^{*}} & O  \\
\end{matrix} \right)+\left\| \left( 1-t \right)S+itT \right\| \right)dt} \\ 
 & \le \left\| S \right\|+\left\| T \right\|. 
\end{aligned}\]
If we let $S=\Re A$ and $T=\Im A$, we get
\[\begin{aligned}
  & \left\| A \right\|+\left| \int\limits_{0}^{1}{\left( \omega \left( \begin{matrix}
   O & \left( 1-t \right)\Re A+it\Im A  \\
   t\Re A+i\left( 1-t \right){\Im A} & O  \\
\end{matrix} \right)-\left\| \left( 1-t \right)\Re A+it\Im A \right\| \right)dt} \right| \\ 
 & \le \int\limits_{0}^{1}{\left( \omega \left( \begin{matrix}
   O & \left( 1-t \right)\Re A+it\Im A  \\
   t\Re A+i\left( 1-t \right)\Im A & O  \\
\end{matrix} \right)+\left\| \left( 1-t \right)\Re A+it\Im A \right\| \right)dt} \\ 
 & \le \left\| \Re A \right\|+\left\| \Im A \right\| \\ 
 & \le 2\omega \left( A \right), 
\end{aligned}\]
as required.
\end{proof}
The following corollary identifies a situation in which the refined bound becomes exact, showing that for operators satisfying $A^{2}=O$ the numerical-radius-based expression coincides with the corresponding norm average, thereby revealing the sharpness of the preceding inequality.

\begin{corollary}
Let $A\in \mathbb B\left( \mathbb H \right)$. If ${{A}^{2}}=O$, then
\[\int\limits_{0}^{1}{\omega \left( \begin{matrix}
   O & \left( 1-t \right)\Re A+it\Im A  \\
   t\Re A+i\left( 1-t \right)\Im A & O  \\
\end{matrix} \right)dt}=\int\limits_{0}^{1}{\left\| \left( 1-t \right)\Re A+it\Im A \right\|dt}.\]
\end{corollary}
\begin{proof}
We know that if ${{A}^{2}}=O$, then $\omega \left( A \right)=\frac{1}{2}\left\| A \right\|$. Thus, from Corollary \ref{11}, we get the desired result.
\end{proof}

\subsection*{Declarations}

\begin{itemize}
\item \textbf{Availability of data and materials}: Not applicable.

\item \textbf{Competing interests}: The authors declare that they have no
competing interests.

\item \textbf{Funding}: Not applicable.

\item \textbf{Authors' contributions}: Authors declare they have contributed
equally to this paper. All authors have read and approved this version.
\end{itemize}

\vskip 0.5 true cm 	

{\tiny (S. Sheybani) Department of Mathematics, Ma.C., Islamic Azad University, Mashhad, Iran 
	
\textit{E-mail address:} shiva.sheybani95@gmail.com}

\vskip 0.3 true cm 	

{\tiny (H. R. Moradi) Department of Mathematics, Ma.C., Islamic Azad University, Mashhad, Iran 
	
\textit{E-mail address:} hrmoradi@mshdiau.ac.ir}

\vskip 0.3 true cm 	

{\tiny (M. Sababheh) Department of Basic Sciences, Princess Sumaya University for Technology, Amman, Jordan
	
\textit{E-mail address:} sababheh@yahoo.com}

%########################


\begin{thebibliography}{99}
\bibitem{Abu-Omar_Rocky_2015} 
A. Abu-Omar, F. Kittaneh, {\it Upper and lower bounds for the numerical radius with an application to involution operators}, Rocky Mountain J. Math. {45}(4) (2015), 1055--1065.

\bibitem{Abu-Omar_LAA_2015} 
A. Abu-Omar, F. Kittaneh, {\it Numerical radius inequalities for $n\times n$ operator matrices}, Linear Algebra Appl. 468 (2015), 18--26.

\bibitem{Gustafson_Book_1997}
K. E. Gustafson, D. K. M. Rao, {\it Numerical range}, Springer, New York, 1997.

\bibitem{Hirzallah_IEOT_2011}
O. Hirzallah, F. Kittaneh, and K. Shebrawi, {\it Numerical radius inequalities for certain $2 \times 2$ operator matrices}, Integral Equ. Oper. Theory. {71} (2011), 129--147. 

\bibitem{Hirzallah_MathScandin_2014} 
O. Hirzallah, F. Kittaneh,  {\it Numerical Radius Inequalities for Several Operators}, Math. Scand. 114(1) (2014), 110--119.

\bibitem{Hirzallah_NFAO_2011}
O. Hirzallah, F. Kittaneh, and K. Shebrawi, {\it Numerical radius inequalities for commutators of Hilbert space operators}, Numer. Funct. Anal. Optim. 32(7) (2011), 739--749.

\bibitem{Kittaneh_Studia_2003}
F. Kittaneh, {\it A numerical radius inequality and an estimate for the numerical radius of the Frobenius companion matrix}, Studia Math. {158}(1) (2003), 11--17.

\bibitem{Kittaneh_Studia_2005}
F. Kittaneh, {\it Numerical radius inequalities for Hilbert space operators}, Studia Math. {168}(1) (2005), 73--80.

\bibitem{Kittaneh_LAA_2015}
F. Kittaneh, M. S. Moslehian, and T. Yamazaki, {\it Cartesian decomposition and numerical radius inequalities}, Linear Algebra Appl. {471} (2015), 46--53. 

\bibitem{Moradi_LAMA_2021} 
H. R. Moradi, M. Sababheh,  {\it More accurate numerical radius inequalities II}, Linear Multilinear Algebra. {69}(5) (2021), 921--933.

\bibitem{Moslehian_MathScand_2017} 
M. S. Moslehian, M. Sattari, and K. Shebrawi, {\it Extensions of Euclidean operator radius inequalities}, Math. Scand. {120}(1) (2017), 129--144.

\bibitem{Sababheh_LAMA_2021}
M. Sababheh, H. R. Moradi, \textit{More accurate numerical radius inequalities (I)}, Linear Multilinear Algebra. {69}(10) (2021), 1964--1973.

\bibitem{1}
M. Sababheh, S. Furuichi, and H. R. Moradi, {\it Operator inequalities via the triangle inequality}, J. Math. Inequal. 18(2) (2024), 631--642.

\bibitem{Sattari_LAA_2015}
M. Sattari, M. S. Moslehian, and T. Yamazaki, {\it Some generalized numerical radius inequalities for Hilbert space operators}, Linear Algebra Appl. {470} (2015), 216--227.


\end{thebibliography}
\end{document}